\documentclass[12pt,oneside]{amsart}
\usepackage{graphicx, amssymb,amscd, verbatim}
\usepackage{euscript}
\usepackage{color}

      \theoremstyle{plain}
      \newtheorem{theorem}{Theorem}[section]
      
      \newtheorem{lemma}[theorem]{Lemma}
      
      \newtheorem{proposition}[theorem]{Proposition}
      \newtheorem{remark}[theorem]{Remark}

\numberwithin{equation}{section}

      \makeatletter
      \def\@setcopyright{}
      \def\serieslogo@{}
      \makeatother

\def\J{{\mathcal J}}

\def\V{\mathcal V}
\def\Vf{{\mathcal V^f}}
\def\VL{{\mathcal V^L}}
\def\Ef{\mathcal E^f}
\def\EL{\mathcal E^L}

\def\M{\mathcal{M}}
\def\E{\mathcal{E}}

\def\R{\mathbb R}
\def\Q{\mathbb Q}
\def\C{\mathbb C}
\def\Z{\mathbb Z}
\def\N{\mathbb N}
\def\T{\mathbb T}

\def\J{\text{Jac}}

\def\QED{\hfill\hfill{\square}}

\begin{document}

\date{\today}
\author{Andrey Gogolev$^\ast$, Boris Kalinin$^{\ast\ast}$, Victoria Sadovskaya$^{\ast\ast\ast}$}

\address{Department of Mathematics, Ohio State University,  Columbus, OH 43210, USA}
 \email{gogolyev.1@osu.edu}

\address{Department of Mathematics, The Pennsylvania State University, 
University Park, PA 16802, USA}
\email{kalinin@psu.edu, sadovska@psu.edu}

\title[Local rigidity of Lyapunov spectrum for toral automorphisms]
{Local rigidity of Lyapunov spectrum for toral automorphisms}

\thanks{$^\ast$ Supported in part by NSF grant DMS-1823150}
\thanks{$^{\ast\ast}$  Supported in part by Simons Foundation grant 426243}
\thanks{$^{\ast\ast\ast}$ Supported in part by NSF grant DMS-1764216}

\begin{abstract}
We study the regularity of the conjugacy between an Anosov automorphism 
$L$ of a torus  and its small perturbation. We assume that $L$ 
has no more than two eigenvalues of the same modulus and that $L^4$ is
irreducible over $\Q$. We consider a volume-preserving $C^1$-small perturbation 
$f$ of $L$. 
We show that if  Lyapunov exponents of $f$ with respect to the volume 
are the same as Lyapunov exponents of $L$, then $f$ is $C^{1+\text{H\"older}}$ 
conjugate to $L$. 
Further,  we establish a similar result for irreducible partially hyperbolic 
automorphisms with two-dimensional center bundle.

\end{abstract}

\maketitle


\section{Introduction and Statements of Results}

Hyperbolic and partially hyperbolic dynamical systems have been one of the 
main objects of study in smooth dynamics. Anosov automorphisms of tori 
are the prime examples of hyperbolic systems. The action of a hyperbolic matrix 
$L\in SL(d,\Z)$ on $\R^d$ induces an automorphism
of the torus $\T^d=\R^d/\Z^d$. It is well-known that any diffeomorphism 
$f$ which is sufficiently $C^1$ close to $L$ is also Anosov and 
topologically conjugate to $L$, i.e. there is a homeomorphism
$h$ of $\T^d$ such that
$$
L = h^{-1} \circ  f \circ h.
$$
The conjugacy  $h$ is unique in the homotopy class of the identity.
It is only H\"older continuous in general, and  various sufficient conditions for $h$
to be smooth have been studied. For  Anosov systems 
with one-dimensional stable and unstable distributions coincidence of eigenvalues 
of the derivatives of return maps at the corresponding periodic points 
was shown to imply smoothness of $h$ \cite{L0,LM, L1,P}.
For  higher dimensional systems even conjugacy of the derivatives of the 
corresponding return maps is not sufficient in general  \cite{L1,L2}. 
Nonetheless, 
smoothness of the conjugacy was established in higher dimensions under various 
additional assumptions \cite{L2,KS03,L3,KS09,G,GKS}.  The last paper establishes 
local rigidity for the most general class of hyperbolic automorphisms. 
\vskip.2cm
\noindent  \cite[Theorem 1.1]{GKS} {\it
Let $L:\T^d\to\T^d$ be an irreducible Anosov automorphism
such that no three of its eigenvalues have the same modulus.
Let  $f$ be a $C^1$-small perturbation of $L$ such that the
derivative $D_pf^n$ is conjugate to $L^n$ whenever $p=f^np$.
Then $f$ is $C^{1+\text{H\"older}}$ conjugate to $L$.
}
\vskip.2cm

We recall that a toral automorphism $L$ is {\em irreducible}\, if it has 
no rational invariant subspaces, or equivalently if its characteristic polynomial 
is irreducible over $\Q$. It follows that all eigenvalues of $L$ are simple.

Recently, R. Saghin and J. Yang obtained smoothness of the conjugacy $h$
for a volume preserving perturbation $f$ of an irreducible $L$ if 
$f$ and $L$ have the same {\em simple} Lyapunov 
exponents with respect to the volume~\cite{SY}.
The main result of this paper is the following theorem, which 
allows for double Lyapunov exponents and extends Saghin-Yang result to 
a much broader class of irreducible hyperbolic automorphisms.

\begin{theorem} \label{main}
Let $L:\T^d\to\T^d$ be an Anosov automorphism such that no three of 
its eigenvalues have the same modulus and its forth power $L^4$ is irreducible.
Let  $f$ be a volume-preserving $C^2$ diffeomorphism of $\T^d$ sufficiently 
$C^1$-close to $L$. If the Lyapunov exponents of $f$ with respect to the volume 
are the same as the Lyapunov exponents of $L$, then $f$ is $C^{1+\text{H\"older}}$ conjugate to $L$.
\end{theorem}
\begin{remark}
In fact, we need $L$ to be irreducible and have no pairs of eigenvalues
of the form $\lambda, -\lambda$ or $i\lambda, -i\lambda$, where $\lambda$ is real.
This follows if  $L^4$ is irreducible.
\end{remark}

We note that the conjugacy of periodic data assumption in  \cite[Theorem 1.1]{GKS}
implies equality of Lyapunov exponents with respect to  the volume
by periodic approximation of Lyapunov exponents \cite[Theorem 1.4]{K11}.
Examples in~\cite{G}  show that irreducibility of $L$ is a necessary in
the results of~\cite{GKS}, and hence in Theorem \ref{main}, except when $L$ is conformal on
the full stable and unstable distributions.

\vskip.1cm
\begin{remark}
The following example demonstrates that irreducibility assumption is more crucial in the current setting when compared to periodic data rigidity. Let $A\colon\T^2\to\T^2$ be a hyperbolic automorphism. Then notice that the automorphism $L\colon (x,y)\mapsto (Ax, Ay)$ and a perturbation $L'\colon (x,y)\mapsto (Ax, Ay+\varphi(x))$ have the same volume Lyapunov spectrum. However, in general, $L$ and $L'$ are not $C^1$ conjugate, because $L'$ may have Jordan blocks in its periodic data. On the other hand, de la Llave established periodic data local rigidity for such $L$~\cite{L2}. 

\end{remark}

\vskip.1cm
\begin{remark}
The toral automorphisms satisfying the assumptions of Theorem~\ref{main} are generic in the following sense.
Consider the set of matrices in $SL(d,\Z)$ of norm at most $T$.
Then the proportion of matrices corresponding to automorphisms
that do not satisfy our assumptions goes to zero as $T \to \infty$.
Moreover, it can be estimated by $c\,T^{-\delta}$ for some $\delta >0$. 
\end{remark}
We refer to~\cite{GKS} for the proof of this remark. One just needs to consider the relation which ensures irreducibility of $L^4$ rather than $L$.

\vskip.3cm

Now we consider the case of a partially hyperbolic toral automorphism, that is an
automorphisms for which {\em some} of the eigenvalues have modulus different from one. 
We call a toral automorphism $L$ {\em totally irreducible} if $L^n$ is irreducible for every 
$n \in \N$. Such an $L$ has no root of unity as an eigenvalues and hence is ergodic and 
partially hyperbolic.
In fact,  \cite[Lemma A.9]{RH} shows that an automorphism $L$ is totally irreducible
if and only if  $L$ is ergodic with  irreducible (over the integers or rationals) characteristic 
polynomial $\chi_L(t)$ that can not be written as a polynomial of $t^n$ for any $n \ge 2$. 
Such an $L$ was called a {\em pseudo-Anosov automorphism} in \cite{RH}.
 
\vskip.1cm

We consider a  totally irreducible automorphism $L$ of $\T^d$
with two-dimensional center bundle, i.e. with exactly two eigenvalues on the unit circle.
It was shown by Rodriguez Hertz in  \cite{RH}  that such $L$ is stably ergodic,
more precisely, for large $N$ any sufficiently $C^{N}$-small volume-preserving  
perturbation of $L$ is also ergodic. 
The following theorem establishes Lyapunov spectrum rigidity for such toral 
automorphisms with simple real stable and unstable eigenvalues.

\begin{theorem} \label{partial}
Let $L:\T^d\to\T^d$ be a totally irreducible automorphism with exactly two eigenvalues 
of modulus one and simple real eigenvalues away from the unit circle.
Let  $f$ be a volume-preserving $C^{N}$-small perturbation of $L$ 
such that the Lyapunov exponents of $f$ with respect to the volume 
are the same as the Lyapunov exponents of $L$. 
If $d >4$ and $N=5$, or $d=4$ and $N= 22$, then $f$ is $C^{1+\text{H\"older}}$ conjugate to $L$.
If $d=4$ and $N=\infty$ then $f$ is $C^\infty$ conjugate to $L$.
\end{theorem}

We note that, in contrast to the hyperbolic case, a small perturbation with different Lyapunov 
spectrum is not necessarily topologically conjugate to $L$.

The theorem above should be compared to~\cite[Theorem D]{SY}, where partially hyperbolic automorphisms with one dimensional center are treated. Such automorphisms leave invariant a fibration by circles. Consequently, one can consider the skew product perturbations which have the same volume Lyapunov exponents. In general, such perturbations are not $C^1$ conjugate the automorphism, in contrast to our Theorem~\ref{partial}. The reason for this is that automorphisms with one dimensional center are reducible, i.e. have eigenvalue $\pm 1$.


\section{Definitions, notations, and outline of the proof of Theorem \ref{main}}

\subsection{Definitions and  notations}
Since $f$ is $C^1$ close to an Anosov automorphism $L$, it is also {\it Anosov}. That is,
there exist a splitting 
of the tangent bundle of $\T^d$ into a direct sum of two $Df$-invariant 
continuous distributions $E^{s,f}$ and $E^{u,f}$,   a Riemannian 
metric on $\M$, and numbers $\nu$ and $\hat\nu$  such that 
\begin{equation}\label{Anosov def}
\|D_xf(v^s)\| < \nu < 1 < \hat\nu <\|D_xf(v^u)\|
\end{equation}
for any $x \in \M$ and unit vectors  $v^s\in E^{s,f}(x)$ and $v^u\in E^{u,f}(x)$.
The distributionss $E^{s,f}$ and $E^{u,f}$ are called {\it stable} and {\it unstable}. 
They are tangent to the stable and unstable foliations 
$W^{s,f}$ and $W^{u,f}$ respectively 
(see, e.g.
\cite{KH}). The leaves of these foliations are as smooth as $f$,
but in general the distributions $E^{s,f}$ and $E^{u,f}$  are only
H\"older continuous transversally to the corresponding foliations.
We denote by $E^{s,L}$ and $E^{u,L}$  the stable and unstable
distributions of $L$.
\vskip.2cm

Let $1<\rho_1 <\dots <\rho_\ell$ be the distinct moduli 
of the unstable eigenvalues of $L$ and let 
$E^{u,L}= E_1^L \oplus E_2^L \oplus \dots \oplus E_\ell^L$
be the corresponding splitting of the unstable distribution.
Since $f$ is $C^1$-close to $L$, the
unstable distribution $E^{u,f}$ splits into a direct sum of $\ell$
invariant H\"older continuous distributions close to the
corresponding ones for $L$:
 $$
   E^{u,f}= E_1^f \oplus E_2^f \oplus \dots \oplus E_\ell^f
$$
\cite[Section 3.3]{Pes}. Since no three eigenvalues of $L$ have the same modulus,
 each $E_i^L$, and hence $E_i^f$, is either one- or two-dimensional. 
The Lyapunov exponent of $L$ on $E_i^L$ is $\chi_i=\log \rho_i$. The assumption that 
the Lyapunov exponents of $f$ with respect to the invariant volume $\mu$ are the 
same as the Lyapunov exponents of $L$ means that for $\mu$-a.e. $x$, 
$$
\lim _{n\to \pm \infty} n^{-1} \log \| Df^n (v)\|= \chi_i=\log \rho_i \quad \text{for all } 0 \ne v \in E_i^f(x).
$$

We also consider the distributions
$E^f_{(i,j)}=E_i^f\oplus E_{i+1}^f\oplus\ldots\oplus E_j^f.$
For any $1 < k \le \ell$,  $\,E_{(k,\ell)}^f$
is a fast part of the unstable distribution and thus it integrates
to a H\"older foliation $W_{(k,\ell)}^f$ with smooth leaves
\cite[Section 3.3]{Pes}.

\vskip.1cm

{\bf Notation.}
 We say that an object is $C^{1+}$ if it is $C^1$ and its differential
  is H\"{o}lder continuous with some positive exponent.
We say that a homeomorphism $h$ is $C^{1+}$ along a foliation
$\mathcal F$ if the restrictions of $h$ to the leaves of $\mathcal
F$ is $C^{1+}$ and the derivative $Dh|_{\mathcal F}$ is H\"older
continuous on the manifold.
\vskip.1cm

For any $1 \le k < \ell$,  $\,E^f_{(1,k)}$ is a slow part of the unstable distribution. 
It also integrates to an $f$-invariant foliation $W_{(1,k)}^f$ with $C^{1+}$ smooth 
leaves. This can be seen by viewing $L$ as a partially hyperbolic map with the 
splitting $E^{s,L}\oplus E_{(1,k)}^L\oplus E_{(k+1,\ell)}^l$. Structural stability of 
partially hyperbolic systems \cite[Theorem 7.1]{HPS} implies that for a $C^1$-small 
perturbation $f$ the ``central" foliation survives; that is, $E^f_{(1,k)}$ integrates to
a foliation $W_{(1,k)}^f$. See \cite[Lemma~6.1]{G} for an alternative proof in our 
setup that also gives unique integrability.

Since both weak and strong flags are uniquely integrable and the leaves of the
corresponding foliations are at least $C^{1+}$, for any $1 \le k \le \ell\,$ the
distribution $E_k^f=E_{(1,k)}^f\cap E_{(k,\ell)}^f$ also integrates uniquely
to a H\"older foliation
$$V_k^f=W_{(1,k)}^f\cap W_{(k,\ell)}^f$$
with $C^{1+}$ smooth leaves. We use analogous notation for the automorphism $L$:
$E_{(i,j)}^L=E_i^L\oplus\ldots\oplus E_j^L, $ and $W_{(i,j)}^L$ and
$V_i^L$ are the linear foliations tangent to $E_{(i,j)}^L$ and
$E_i^L$ respectively. \vskip.2cm

\vskip.2cm

\subsection{Outline of the proof of Theorem \ref{main}}
By the Structural Stability Theorem, there exists a unique
bi-H\"older continuous homeomorphism $h$ of $\T^d$ close to the
identity in $C^0$ topology such that
  $$
    h \circ L = f \circ h.
  $$
First, for any sufficiently $C^1$-small perturbation of an Anosov automorphism 
of $\T^d$, the conjugacy $h$ takes the flag of weak foliations  for $L$ into
the corresponding  weak flag for $f$:

\begin{lemma} [cf. Lemma~6.3 in~\cite{G} and Lemma 2.1 in \cite{GKS}]
\label{weak_flag_pres} $\;$\\
For any $1\le k\le \ell,$ $\,h(W_{(1,k)}^L )=W_{(1,k)}^f$.
\end{lemma}

Theorem~\ref{main} is proved by showing inductively that $h$ is $C^{1+}$ 
along $W_{(1,k)}^L$ for any $k$ and thus along $W_{(1,\ell)}^L=W^u(L)$. 
By the same argument, $h$ is $C^{1+}$ along $W^s(L)$ and hence $h$ 
is $C^{1+}$ by Journ\'e Lemma:
\begin{lemma}[Journ\'e~\cite{J}]  \label{Journe}
Let $\M_j$ be a manifold and $\mathcal F_j^s$, $\mathcal F_j^u$ be
continuous transverse foliations on $\M_j$ with uniformly smooth
leaves, $j=1, 2$. Suppose that $h:\M_1\to \M_2$ is a homeomorphism
that maps $\mathcal F_1^s$  into $\mathcal F_2^s$ and $\mathcal
F_1^u$ into $\mathcal F_2^u$. Moreover, assume that the restrictions
of $h$ to the leaves of these foliations are uniformly $C^{r+\alpha}$,
$r\in \mathbb N$, $0<\alpha<1$. Then $h$ is $C^{r+\alpha}$.
\end{lemma}

The inductive step is given by the following proposition \cite[Proposition 2.4]{GKS}:

\begin{proposition}\label{fast_to_fast}
Suppose that  $h(V_i^L)=V_i^f$, $1\le i\le k-1$, and $h$ is a
$C^{1+}$ diffeomorphism along $W^L_{(1,k-1)}$. Then $h(V_k^L)=V_k^f$
and  $h$ is a $C^{1+}$ diffeomorphism along $W_{(1,k)}^L$.
\end{proposition}

The first part of the proof is to establish that $h(V_k^L)=V_k^f$. This part
is identical to the proof in \cite{GKS}. The second part, that $h$ is a $C^{1+}$ 
diffeomorphism along $W_{(1,k)}^L$, follows from Journ\'e Lemma and the
following proposition:

\begin{proposition} \label{smooth along V}
If $\,h(V^L_i)=V^f_i$, then
 $h$ is a $C^{1+}$ diffeomorphism along $V_i^L$.
\end{proposition}

Since $V_1^L=W_{(1,1)}^L$, Lemma~\ref{weak_flag_pres} implies that 
$h(V_1^L)=V_1^f$, and then Proposition~\ref{smooth along V} yields  
that $h$ is $C^{1+}$ along $V_1^L$. This provides the base of the induction. 
Thus to prove Theorem \ref{main} it remains to establish Proposition~\ref{smooth along V},
which is done in Section~\ref{smooth along V proof}. 

\vskip.4cm

\section{Proof of Proposition~\ref{smooth along V}}
\label{smooth along V proof}
In this section we  write
$$
  \VL \overset{\text{def}}{=}V_i^L,  \quad
  \Vf\overset{\text{def}}{=}V_i^f, \quad
  \EL\overset{\text{def}}{=}E_i^L=T\VL, \quad
  \Ef\overset{\text{def}}{=}E_i^f=T\Vf.
$$
The key ingredients of the proof are establishing conformality of $f$ along $\Vf$
and showing that the Jacobian  of $f$ along $\Vf$ is cohomologous to a constant.
They require arguments different from those in \cite{GKS}, where periodic data
was used in an essential way.
After this we show that $h$ is Lipschitz along $\VL$ as a limit of smooth
maps with uniformly bounded derivatives. Then we prove that the
measurable derivative of $h$ along $\VL$ is actually H\"older
continuous. We consider the case when $\Vf$ is two-dimensional, as
in the one-dimensional case the conformality is trivial and the other
arguments are simpler.

\subsection{Conformality of $Df|_{\Ef}$.}
Since the linear map $L$ is irreducible, it is diagonalizable over $\C$. Therefore, 
as the eigenvalues of $L|_{\EL}$ have the same modulus, $L|_{\EL}$ is conformal 
with respect to some norm on $\EL$. 

We view the restriction $F=Df|_{\Ef}$
as a linear cocycle over $f$, that is an automorphism of the vector bundle $\E^f$ which 
covers $f$. 
We will now show that $F$  is conformal with
respect to some H\"older  continuous Riemannian norm $\|\cdot\|^f$ on $\E^f$,
 that is the linear map $F_x=Df|_{\Ef(x)} : \E_x^f \to \E_{fx}^f$ is conformal for every $x$.
Since the bundle $\E^f$ is H\"older continuous, 
the cocycle $F$ is $\beta$-H\"older for  some $\beta>0$.
 Also, since  $f$ is close to $L$, the cocycle $F$ is close to 
conformal and therefore {\it fiber bunched}, that is, for all $x\in\T^d$,
$$
\|F_x\|\cdot \|F_x^{-1}\|\cdot \nu^\beta < 1 \;\text{ and }\;
\|F_x\|\cdot \|F_x^{-1}\|\cdot  \hat \nu^\beta < 1, 
$$
where $\nu, \hat\nu$ are as in \eqref{Anosov def}. Also, by the assumption, $F$ has only 
one Lyapunov exponent with respect 
to the volume $\mu$.
Thus we can apply the following proposition.

\begin{proposition}
[Two-dimensional Continuous  Amenable Reduction] \label{reduction} $\;$\\
Let $f$ be a $C^{1+\text{H\"older}}$ Anosov diffeomorphism of a compact manifold $\M$
preserving a volume $\mu$.
Let $\E$ be a vector bundle over $\M$ with two-dimensional fibers, and 
let $F:\E\to \E$ be a H\"older continuous  fiber bunched cocycle over $f$ with one 
Lyapunov exponent with respect to $\mu$. 
Then at least one of the following holds:
 \begin{enumerate}
\item $F$ is conformal with
respect to a H\"older  continuous Riemannian norm on $\E$;
 
\item  $F$ preserves a H\"older continuous  one dimensional sub-bundle;
 
\item $F$ preserves a H\"older continuous field of two transverse lines.
 \end{enumerate}
  \end{proposition} 
  
  \begin{proof}
  We note that (1) is equivalent to  preserving a H\"older continuous conformal structure on $\E$.
  
The proposition essentially follows from the Continuous Amenable Reduction established 
in \cite{KS13}, specifically Theorem 3.4,  Corollary 3.8, and the first remark after Theorem 3.4. 
The conclusion for 2-dimensional fibers yields that $F$ satisfies (1) or (2) or 
\begin{itemize}
\item[(3$'$)] There exists a finite cover $\tilde F : \tilde  \E \to \tilde  \E$ of $F$ which
 preserves a union of $j \ge 2$ distinct at every point H\"older continuous one dimensional sub-bundles 
 $U=E^1\cup ... \cup E^j$, and the projection of $U$ to $\E$ is invariant under $F$.
\end{itemize}
If $j=2$, then this $U$ is exactly the invariant H\"older continuous field of two lines.
We claim that if $j>2$ then, in fact, (1) holds. Indeed, by passing to another cover $\hat  \E$
if necessary, we can make $E^1$ and $E^2$ orientable and pick nonzero vector 
fields $v_1(x) \in E^1$ and $v_2(x) \in E^2$.
In this basis the lift $\hat F : \hat  \E \to \hat  \E$ becomes a H\"older continuous diagonal 
$GL(2,\R)$-valued function $A(x)=\text{diag}(a_1(x),a_2(x))$. 
Then existence of another invariant sub-bundle 
$E^3$ implies that the functions $a_1$ and $a_2$ are H\"older cohomologous 
\cite[Lemma 7.1]{S13}, that is $a_1(x)/a_2(x)=\phi(fx)/\phi(x)$ for some  
H\"older continuous function $\phi$. It follows that $A$ is uniformly quasiconformal, 
that is  $\|A^n_x\| \cdot \|(A^n_x)^{-1}\|$ is uniformly bounded in $x$ and $n$, 
where $A^n_x = A(f^{n-1}x) \cdots A(fx) A(x)$. 
This implies that $\tilde F$ and $F$ are also uniformly quasiconformal and 
hence $F$ is conformal with respect to a H\"older  continuous Riemannian norm 
on $\E$ by \cite[Corollary 2.5]{KS10}. 
\end{proof}
 
 Now we want to eliminate possibilities (2) and (3) for our cocycle $F$. Recall 
 that $L$ has two eigenvalues of the same modulus on $\EL$. Since $L^4$ is irreducible
 and thus has simple eigenvalues we conclude that this is a pair of complex 
conjugate eigenvalues on $\EL$ different from $i\lambda$, $-i\lambda$. Hence the projective
action of $L|_{\EL}$ on the lines through $0$ does not have a fixed point or an invariant two-point set. 
This property persists for any sufficiently close linear map, and thus holds for $F_p=Df|_{\Ef(p)}$ 
at the fixed point $p=f(p)$ close to $0$. Hence $F$ can not satisfy (2) or (3) and thus must satisfy (1).

\vskip.5cm

\subsection{Jacobian along $\Vf$} \label{Jacobian}
In this section we show that the Jacobian of $f$ along $\Vf$ is cohomologous to a constant.
 This result can be deduced using \cite[Corollary G]{SY}.
Our argument is based on the same idea, but adapted to our setting becomes much shorter, 
so we include it for reader's convenience.

Denote by $m$ the standard Lebesgue measure on $\T^d$, which is preserved by $L$,  
and by $\mu$ the smooth $f$-invariant measure. Then $\mu=h_*(m)$ since both coincide with the 
unique measure of maximal entropy for $f$. Indeed, $h_*(m)$ has maximal entropy since
$m$ has maximal entropy for $L$, and $\mu$ has the same entropy since by the Pesin formula
it equals the sum of positive Lyapunov exponents for $f$, which are the same as 
for $L$.

By the assumption of the proposition we have $h(\VL)=\Vf$.
While $\Vf$,  a priori, is not necessarily an absolutely continuous foliation,
we claim that the conditional  measures of $\mu$ on the leaves of $\Vf$ are absolutely 
continuous. This follows from the assumption 
on the Lyapunov exponents and Lemma \ref{Ledrappier} below.

Let $\V$  be an expanding foliation for $f$. We say that a measurable partition $\xi$ is
subordinate to $\V$ if it  satisfies the following properties.
\begin{enumerate}
\item For all $x$, $\xi(x)\subset \V(x)$ and for a.e. $x$ partition element $\xi(x)$ is bounded and contains a neighborhood of $x$ in $\V(x)$;
\vskip.05cm
\item $\vee_{i\ge 0} f^{-i}\xi$ is the partition into points;
\vskip.1cm
\item $f\xi<\xi$.
\end{enumerate}

\begin{lemma} (Ledrappier \cite{L}, \cite[Lemma 5.5]{SY}) \label{Ledrappier} 
Let $\V$ be an expanding foliation for $f$, let $\mu$ be an invariant measure, and let $\xi$ be a measurable partition subordinate to $\V$. Then the conditional measures of $\mu$ are absolutely continuous on leaves of $\V$ if and only if the following formula for conditional entropy holds
$$
 H_\mu(f^{-1}\xi | \xi)=\int \log \J\, (f|_\V)\, d\mu. 
 $$
\end{lemma}
We consider a measurable partition $\xi$ subordinate to $\V=\Vf$, which exists for example by \cite{LS}. 
We let $\xi'=h^{-1}\xi$. Since $h_*m=\mu$, applying the lemma to $L,\VL,\xi'$ and $m$ we get
$$
 H_\mu(f^{-1}\xi | \xi)= H_m(L^{-1}\xi' | \xi')=\int \log \J (L|_{\VL})\, dm=
 \int \log \J (f|_{\Vf})\, d\mu.
$$
The last equality holds since the integrals are equal to the sums of the Lyapunov exponents 
of $L$ and $f$ corresponding to the  foliations $\VL$ and $\Vf$, which are the same by the assumption.
Applying the lemma to $f,\Vf, \xi$, and $\mu$ we conclude that the conditional measures of $\mu$ on the leaves of $\Vf$ are absolutely continuous.
\vskip.1cm

We denote by $\{\mu_x=\mu_{\Vf(x)}:\,x\in \T^d\}$ the system of conditional measures for $\mu$ 
on the leaves of $\Vf$ obtained by pushing forward by $h$ the standard volume on $\VL$. It is 
defined for all $x\in \T^d$, satisfies $\mu_{y}=\mu_x$ for all $y \in \Vf (x)$,  and is pushed by $f$ as
$$
f_*(\mu_x)=k \,  \mu_{fx},\quad\text{where }k^{-1}= \det L|_{\VL}.
$$
Let $\sigma _x$ the volume on $\Vf(x)$ induced by the standard metric on $\T^d$.
Then absolute continuity of the conditional measures of $\mu$ yields that there exists a
measurable function $\rho$ on $\T^d$ such that  $\mu_x = \rho \, \sigma _x$ for $\mu$-a.e. $x$. 
This function is given by the Radon-Nikodym derivative
$$
\rho (y)= \frac {d\mu_{x}}{d\sigma _{x}}(y)=\lim_{r\to 0}\, \frac{\mu_x (B^\Vf (y,r))}{\sigma_x(B^\Vf (y,r))}
\quad\text{for a.e. }y \in \Vf (x).
$$
We denote  the Jacobian of $f$ with respect to the volumes $\sigma _x$ and $\sigma _{fx}$ by
$$
J(x)=\J (f|_\Vf)(x)=\frac {d\sigma _{fx}}{d(f_*\sigma _{x})}(x).
$$
The function $J(x)$ is  H\"older continuous on $\T^d$ and satisfies $f_*(\sigma_x)=J^{-1} \,  \sigma_{fx}$.
Then  $f_*(\mu_x)=k \,  \mu_{fx}$ yields that $\rho (x)=k J(x) \rho(fx)$ for $\mu$-a.e. $x$. Indeed,
$$
(\rho \circ f^{-1}) f_*( \sigma_x)=
 f_*(\rho \sigma_x) = f_*(\mu_x) =k \,  \mu_{fx}= k \rho  \,\sigma _{fx} 
= ( J\circ f^{-1}) k \rho f_*(\sigma _{x}).
$$
We also note that $\rho (x)>0$ for $\mu$-a.e. $x$ since $\mu$ is equivalent to the  standard volume.
Thus we have 
$$
\det L|_{\VL} = J(x) \rho(fx) \rho (x)^{-1}\quad\text{for }\mu\text{-a.e. }x,
$$
 that is, $J$ is measurably 
cohomologous to the constant $ \det L|_{\VL}$. By the measurable Liv\v{s}ic theorem for scalar 
cocycles \cite{Liv}, $\rho$ coincides $\mu$-a.e. with a H\"older continuous function on $\T^d$.

\subsection{The conjugacy $h$ is Lipschitz along $\Vf$}
The proof is an adaptation of arguments in \cite{L2,GKS}.
Let $\bar{\V}^L$ be the linear integral
foliation of
$$
E^{s,L}\oplus E_1^L\oplus\ldots\oplus E_{i-1}^L\oplus
E_{i+1}^L\oplus\ldots\oplus E_l^L.
$$
We define a map
$h_0$ by intersecting  local leaves:
$$
h_0(x)=\V^{f,loc}(h(x))\cap\bar{\V}^{L, loc}(x).
$$
It is easy to check (see \cite{GKS}) that $h_0$ is well-defined,  close to $h$, and
satisfies
\begin{enumerate}
\item $h_0(\VL)=\Vf$, moreover, $h_0(\VL(x))=\Vf(h(x))$ for all $x$ in $\T^d$;
\vskip.1cm
\item $\sup_{x\in\mathbb T^d} d_{\Vf}(h_0(x),h(x))<+\infty$,
where $d_{\Vf}$  is the distance along the leaves;
\vskip.1cm
\item $h_0$ is $C^{1+}$ diffeomorphism along the leaves of $\VL$;
\vskip.1cm
\item $h=\lim_{n\to\infty}h_n$ uniformly on $\T^d$, where $h_n=f^{-n}\circ h_0\circ L^n.$
\end{enumerate}

\vskip.2cm

To prove that $h$ is Lipschitz along $\Vf$ we show that the derivatives
of the maps $h_n$ along $\VL$ are uniformly bounded. With the notation $F^n_x=Df^n|_{\Ef(x)}$  we obtain
$$
\begin{aligned}
 \|D_{\VL(x)}h_n\|
 & \le \| (F^n_{f^{-n}(h_0(L^nx)))})^{-1}\|\cdot \|D_{\VL(L^nx)}h_0\| \cdot \|L^n|_{\EL} \| \\
 & \le \| (  F^n_{h_n(x))}  )^{-1}\| \cdot \|L^n|_{\EL} \|
 \cdot  \sup\{ \|D_{\VL(z)}h_0\|: z\in\T^d\}. \\
 \end{aligned}
 $$
Since  $D_{\VL}h_0$ is continuous on $\T^d$, the last term is finite and it remains to show 
that $\| (  F^n_{x}  )^{-1}\|\cdot \|L^n|_{\EL} \|$ is uniformly bounded
in $x\in \T^d$ and $n\in \N$.

\vskip.1cm
We denote by $\|\cdot\|$ a norm on $\EL$ for which $L|_{\EL}$ is conformal,
and  by $\| \cdot \|_x^f$ a H\"older continuous Riemannian norm on $\Ef$
for which $F=Df|_{\Ef}$ is conformal. Then
$$
    \| F(v) \|_ {f(x)}^f = a(x) \, \|v\|_x^f \;
    \text{ for any } x\in \T^d,\;v\in \Ef(x),\, \text{ where } a(x)=(\J\, f|_{\Vf (x)})^{\frac 12}.
$$
Hence for the operator norm with respect to $\| \cdot \|^f$ we have 
$$
\| (F_{x})^{-1}\|^f = \left( \| F_{x}\|^f\right)^{-1}= a(x)^{-1} .
$$
Since $\J f|_{\Vf}$ is  H\"older  cohomologous to the constant $\det L|_{\EL}$,
we conclude that $a(x)$ is H\"older  cohomologous to the constant $b=  \|L|_{\EL}\|$,
i.e. $ b/a(x)=\phi(Lx)/\phi(x).$ for some H\"older continuous function $\phi:\T^d\to\R_+$. We conclude  that
$$
\|L^n|_{\EL} \| \cdot \| (F^n_{x})^{-1}\|^f =
b^n (a(f^{n-1}x) \cdots a(fx) a(x))^{-1} =  \phi(L^nx) /\phi(x)
$$
is uniformly bounded since $\phi$ is continuous on $\T^d$. Since the
norm $\| \cdot \|^f$ is  equivalent to $\|\cdot\|$ we obtain that
$\| (F_x)^{-1}\|\cdot \|L^n|_{\EL} \|$  is uniformly
 bounded in $y$ and $n$. We conclude that $\|D_{\VL(x)}h_n\| $
 is uniformly bounded in $x$ and $n$, and hence $h=\lim_{n\to\infty}h_n$ is
 Lipschitz along $\Vf$.

A similar argument shows that
$\|(D_{\VL(x)}h_n)^{-1}\|$ is uniformly bounded and hence
$h$ is bi-Lipschitz along $\Vf$. In particular, $D_{\VL}h$ exists and is
invertible almost everywhere with respect to the volume.

\vskip.2cm

\subsection{The conjugacy $h$ is $C^{1+}$ along $\Vf$}

Differentiating $f\circ h= h \circ L$ along $\VL$ on
a set of full Lebesgue measure we obtain
$$
 F_{h(x)} \circ D_{\VL(x)}h = D_{\VL(Lx)} h \circ L |_{\EL(x)} ,
$$
i.e., the cocycles $F_{h(x)}$ and $L |_{\EL(x)}$ are
cohomologous with transfer function $D_{\VL(x)}h$. The bundle $\Ef$
is trivial since it is close to the trivial bundle $\EL$. Therefore,
$F_{h(x)}$ and $L |_{\EL(x)}$ can be viewed as  H\"older
continuous $GL(2,\R)$-valued cocycles over the automorphism $L$.
While in general measurable transfer functions are not
necessarily continuous \cite[Section 9]{PW}, for conformal
cocycles the measurable transfer function coincides almost
everywhere with a H\"older continuous one by \cite[Theorem 2.7]{S15}. 
We conclude that $D_{\VL(x)}h$ is H\"older continuous, and hence  $h$
is a $C^{1+}$ diffeomorphism along $\VL$. $\QED$


\section{Proof of Theorem \ref{partial}}

Since $f$ is $C^1$-close to $L$, it is {\em partially hyperbolic}, more precisely,
there exist a nontrivial $Df$-invariant splitting $E^s\oplus E^c \oplus E^u$ 
of the tangent bundle of $\T^d$,
 a Riemannian metric on $\T^d$, and 
constants $\nu<1,\,$ $\hat\nu >1,\,$ $\gamma,$ $\hat\gamma\,$ such that 
for any $x \in \M$ and unit vectors  
$\,v^s\in E^{s,f}(x)$, $\,v^c\in E^{c,f}(x)$, and $\,v^u\in E^{u,f}(x)$,
$$
\|D_xf(v^s)\| < \nu <\gamma <\|D_xf(v^c)\| < \hat\gamma <
\hat\nu <\|D_xf(v^u)\|.
$$
 The sub-bundles $E^{s,f}$, $E^{u,f}$, and $E^{c,f}$ are called, respectively,
 stable, unstable, and center.
The  stable and unstable sub-bundles are tangent to the stable and unstable
foliations $W^{s,f}$ and $W^{u,f}$, respectively. The leaves of these  foliations are
as smooth as $f$.
Moreover, as $f$ is a $C^1$-small perturbation of $L$, the center bundle $E^{c,f}$
is tangent to  a foliation $W^{c,f}$ with $C^{1+}$ leaves.

\vskip.1cm

Since the center Lyapunov exponents of $f$ are the same, we can use the following result
by Avila and Viana to conclude that $f$ is not accessible.
A partially hyperbolic  diffeomorphism $f$  is called {\em accessible}  if any two points 
 in $\M$ can be connected by an $su$-path, that is,
by a concatenation of finitely many subpaths which lie entirely in a single 
 leaf of $W^s$ or  $W^u$. 
\vskip.2cm

\noindent \cite[Theorem 8.1]{AV}  
{\it Let $L$ be as in Theorem \ref{partial}. Then there exists a neighborhood $U$ 
of $L$ in the space of $C^N$ volume preserving diffeomorphisms of $\T^d$ 
such that if  $f\in U$ is accessible then its center Lyapunov exponents are distinct.}
\vskip.2cm 

For the perturbation $f$ which is not accessible, it was proved in  \cite[Section 6]{RH} 
(cf. \cite[Remark 8.3]{AV}) that  
$f$ and $L$ are conjugate by a bi-H\"older homeomorphism $h$ which is $C^{1+\text{H\"older}}$
 along the leaves of the center foliation.
 In the above theorem, the regularity $N$ is the one needed to apply the KAM-type results in \cite[Section 6]{RH}, which is exactly  the regularity that  we require in Theorem \ref{partial}.

Now we claim that $h$ preserves the volume, that is,  $h_*(m)=\mu$. The argument is the 
same as for Anosov case: both $h_*(m)$ and $\mu$ are  measures of maximal entropy  
and hence coincide by uniqueness. The uniqueness of the measure of maximal entropy for 
$L$ is well-known (\cite{B}, see also~\cite[Theorem 2.6]{Sch2}). 
The uniqueness for $f$ follows since $f$ and $L$ are topologically conjugate. 


Now use an approach similar to the Anosov case, however now the foliations $V_i^f$ are one-dimensional.  For each $i$ the conditional 
measures of $\mu$ along $V_i^f$ are absolutely continuous by the  argument in Section \ref{Jacobian}.  Therefore, their densities are given, up to normalization, by the ratio of the Jacobians \cite[Proposition 2.3 A4]{SY}:
$$
\frac {d\mu_{x}}{d\sigma _{x}}(y)= c \lim_{n\to \infty}\, \frac{\J (f^{n}|_\Vf)(y)}{\J (f^{n}|_\Vf)(x)}.
$$
Since $\J (f |_\Vf)(x)$ is H\"older continuous and the foliation $\Vf$ is contracting, 
it follows that the densities are H\"older continuous on $\T^d$.
This implies that Proposition \ref{smooth along V} holds  for each $i$, that is, if $\,h(V^L_i)=V^f_i$ 
then $h$ is a $C^{1+}$ diffeomorphism along $V_i^L$. Indeed, the conjugacy $h$ 
maps the conditional measures on the leaves of $V_i^L$ to those on the leaves of $V_i^f$
and hence $h^{-1}$ is obtained by integrating the H\"older continuous density along the 
one-dimensional leaves.
\vskip.1cm
We note that $h(W^{u,L}) = W^{u,f}$ and $h(W^{s,L}) = W^{s,f}$. Indeed, 
$$
d(f^nx,f^ny) \to 0 \;\Longrightarrow \; d(L^nh^{-1}(x),L^nh^{-1}(y)) \to 0 
\;\Longrightarrow \; h^{-1}(y) \in W^{s,L}(h^{-1}(x))
$$
Thus $h^{-1}(W^{s,f}(x)) \subset W^{s,L}(h(x))$ and equality follows as $h$ is a homeomorphism.

Now the same inductive process as in the Anosov case shows that $h$ is a $C^{1+}$ 
diffeomorphism along $W^{u,L}$, and similarly is  $C^{1+}$ along $W^{s,L}$. We recall that
$h$ is smooth along the leaves of $W^{c,L}$. Applying Journ\'e  Lemma \ref{Journe} twice we obtain
that $h$ is $C^{1+}$ along $W^{u,L} \oplus W^{u,L}$ and then  on $\T^d$.
\vskip.3cm

If $f$ is $C^\infty$ close to $L$ then \cite[Section 6]{RH} gives that $h$ is also $C^\infty$ 
along along the center foliation. If $d=4$ then $E^s$ and $E^u$ are one-dimensional and 
their leaves are $C^\infty$ manifolds. In this case the densities of the conditional measures
are $C^\infty$ along the leaves because they are given by the infinite product of ratios of Jacobians. Then $h$ is $C^\infty$ 
along these foliations and hence  is $C^\infty$ on $\T^d$ by Journ\'e Lemma.
$\QED$


\end{document}